\documentclass{amsart}

\usepackage{graphicx}
\usepackage{amssymb}
\usepackage{amsmath}
\usepackage{amsthm}
\usepackage{cite}
\usepackage{cases}
\usepackage[font={it}, margin=1cm]{caption}

\newtheorem{theorem}{Theorem}[section]
\newtheorem{lemma}[theorem]{Lemma} 
\newtheorem{proposition}[theorem]{Proposition} 
\newtheorem{corollary}[theorem]{Corollary} 
\newtheorem{thmletter}{Theorem}
\newtheorem{corolletter}[thmletter]{Corollary}

\newtheoremstyle{example}{3pt}{3pt}{}{10pt}{\itshape}{:}{.5em}{}
\theoremstyle{example}
\newtheorem{example}[theorem]{Example}

\newcommand{\p}[1]{\noindent {\newline\bf #1.}}
\newcommand{\pf}[1]{\indent {\newline\it #1:}}
\newcommand{\out}{\operatorname{Out}}
\newcommand{\aut}{\operatorname{Aut}}
\newcommand{\inn}{\operatorname{Inn}}
\newcommand{\GL}{\operatorname{GL}}
\newcommand{\BS}{\operatorname{BS}}
\newcommand{\<}{\langle}
\renewcommand{\>}{\rangle}
\newcommand{\stab}{\operatorname{Stab}}

\usepackage{verbatim}

\begin{document}

\title[$\out(G)$ for $2$-generator, $1$-relator groups with torsion]{The outer automorphism groups of two-generator, one-relator groups with torsion}
\author[A.D.Logan]{Alan D. Logan}

\address{University of Glasgow, School of Mathematics and Statistics, University Gardens, G12 8QW, Scotland}
\email{Alan.Logan@glasgow.ac.uk}
\thanks{The author would like to thank his PhD supervisor, Stephen J. Pride, and Tara Brendle for many helpful discussions about this paper.}

\subjclass[2010]{Primary 20F28, 20E99, secondary 20F65, 20F67}

\keywords{One relator groups, Outer automorphism groups, Nielsen equivalence}

\date{\today}

\maketitle

\begin{abstract}
The main result of this paper is a complete classification of the outer automorphism groups of two-generator, one-relator groups with torsion.
To this classification we apply recent algorithmic results of Dahmani--Guirardel, which yields an algorithm to compute the isomorphism class of the outer automorphism group of a given two-generator, one-relator group with torsion.
\end{abstract}

\section{Introduction}

Let $R$ denote a non-trivial freely and cyclically reduced word in $F(a, b)$ which is not a proper power of any other word (we maintain this convention throughout the paper). The group $G=\< a, b; R^n\>$ has torsion if and only if $n>1$~\cite[Theorem~4.12]{mks}, and we call such a group $G$ a \emph{two-generator, one-relator group with torsion}.

In this paper we classify the outer automorphism groups of two-generator, one-relator groups with torsion. A \emph{primitive element} of $F(a, b)$ is a word $R$ which is contained in a basis of $F(a, b)$, and this paper studies the case when $R$ is not primitive. This focus is because $R$ is primitive if and only if $G=\<a, b;R^n\>\cong \mathbb{Z}\ast C_n$ \cite{Pride1977}, and if and only if $G=\<a, b;R^n\>$ is not one-ended \cite{fischer1972one}.

\p{One-relator groups}
One-relator groups are classically studied (see, for example, the classic texts in combinatorial group theory \cite{mks} \cite{L-S} or the early work of Magnus \cite{magnus1930freiheitssatz} \cite{magnus1932wordproblem}), and these groups continue to have applications to this day, for example in $3$-dimensional topology and knot theory~\cite{2014Ichihara}~\cite{Friedl2015Two}. One-relator groups with torsion are hyperbolic and as such they serve as important test cases for this larger class of groups. For example, the isomorphism problem for two-generator, one-relator groups with torsion was shown to be soluble \cite{Pride1977} long before Dahmani--Guirardel's recent resolution of the isomorphism problem for all hyperbolic groups \cite{dahmani2011isomorphism}. As another example, Wise recently resolved the classical conjecture of G. Baumslag that all one-relator groups with torsion are residually finite \cite{wise2012riches}, while it is still an open question as to whether all hyperbolic groups are residually finite.

Two-generator, one-relator groups with torsion are particularly important within the class of one-relator groups with torsion for two reasons. Firstly, every one-relator group with torsion can be embedded into a two-generator, one-relator group with tosion. Secondly, every two-generator subgroup of a one-relator group with torsion is either a free group or a two-generator, one-relator group with torsion \cite{pride1977twogensubgps}.

\p{Main theorem}
Writing $D_{n}$ for the dihedral group of order $2n$, the main result of this paper is as follows. Recall that $R$ denotes a non-trivial word in $F(a, b)$ which is not a proper power.
\begin{thmletter}
\label{thm:onerelmaintheorem}
Let $G$ be a two-generator, one-relator group with torsion, so $G\cong\langle a, b; R^n\rangle$ with $n>1$.
\begin{enumerate}
\item If $G\cong \langle a, b; [a, b]^n\rangle$ then $\out(G)\cong \GL_2(\mathbb{Z})$.
\item Suppose that $G$ is one-ended and $G\not\cong\langle a, b; [a, b]^n\rangle$.
\begin{enumerate}
\item\label{mainthm:ref1} If $\out(G)$ is infinite then it is isomorphic to $D_{\infty}\times C_2$, $D_{\infty}$, $\mathbb{Z}\times C_2$ or $\mathbb{Z}$.
\item\label{mainthm:ref2} If $\out(G)$ is finite then it embeds into either $D_6$ or $D_4$.
\end{enumerate}
\item If $G$ is infinitely ended, so $G\cong \mathbb{Z}\ast C_n$, then $\out(G)~\cong~D_{n}\rtimes\aut(C_n)$.
\end{enumerate}
Moreover, all the possibilities from (\ref{mainthm:ref1}) and from (\ref{mainthm:ref2}) occur.
\end{thmletter}
Explicit examples of each of the groups occurring in (\ref{mainthm:ref1}) are given in Lemma~\ref{lem:realiseinfinite}, and examples of each of the groups occurring in (\ref{mainthm:ref2}) are given in Lemmas~\ref{lem:realisefinitebatch1}~and~\ref{lem:examplesFiniteOut}. If $G$ is one-ended then our methods show that the isomorphism class of $\out(G)$ depends only on the root $R$ of the relator $R^n$.

A general theory exists for the outer automorphism groups of hyperbolic groups, based on JSJ-decompositions. However, this theory is limited as it describes only a finite-index subgroup of the outer automorphism group \cite{levitt2005automorphisms}. In Theorem~\ref{thm:JSJdecomp} we prove the relevant coarse description (our proof uses fixed points of automorphisms in free groups rather than JSJ-decompositions). The proof of Theorem~\ref{thm:onerelmaintheorem} improves upon the coarse description of Theorem~\ref{thm:JSJdecomp} by combining it with faithful linear representations over $\mathbb{Z}$ of the outer automorphism groups, which we prove exist in Theorem~\ref{injecting}.

In Section~\ref{sec:applications} we give two applications of Theorem~\ref{thm:onerelmaintheorem}. Our first application combines Theorem~\ref{thm:onerelmaintheorem} with recent work of Dahmani--Guirardel on algorithms in hyperbolic groups \cite{dahmani2011isomorphism} to prove the following corollary of Theorem~\ref{thm:onerelmaintheorem}.
\begin{corolletter}
\label{corol:MainCorollary}
There exists an algorithm which takes as input a presentation $\langle a, b; R^n\rangle$, $n>1$, defining a group $G$ and gives as output the isomorphism class of $\out(G)$.
\end{corolletter}
The algorithm of Corollary~\ref{corol:MainCorollary} does not simply find the isomorphism class of $\out(G)$ but also finds automorphisms of $G=\langle a, b; R^n\rangle$, given in terms of their action on the generators $a$ and $b$, whose cosets generate $\out(G)$.

Our second application in Section~\ref{sec:applications} is an explanation of how to use Theorem~\ref{thm:onerelmaintheorem} to write down a presentation for the full automorphism group of a two-generator, one-relator group with torsion.

\p{Previous results} Much of the previous work on the outer automorphism groups of one-relator groups with torsion is based around residual finiteness. For example, Kim--Tang proved certain specific two-generator, one-relator groups with torsion have residually finite outer automorphism groups \cite{paper2, paper3}, while it is a recent result of Carette that an (arbitrarily-generated) one-relator group with torsion has residually finite outer automorphism group \cite{carette2013virtually}.%
\footnote{Note that Carette uses the deep result of Wise that one-relator groups with torsion are residually finite \cite{wise2012riches}.}
Note that, restricting to the two-generator case, Theorem~\ref{thm:onerelmaintheorem} is much stronger than these previous results.

It is worth mentioning that Theorem~\ref{thm:onerelmaintheorem} shows that the outer automorphism groups of two-generator, one-relator groups with torsion are very similar to the outer automorphism groups of one-relator groups with non-trivial center \cite{GHMR} and to the outer automorphism groups of Baumslag--Solitar groups $\BS(m, n)=\langle b, s; s^{-1}b^ms=b^n\rangle$ with $|m|\geq|n|$ and $n$ is not a proper divisor of $m$ \cite{clay2006deformation}. In each of these two cases, the base groups are non-hyperbolic two-generator, one-relator groups \emph{without} torsion, but their outer automorphism groups are found in the list given in Theorem~\ref{thm:onerelmaintheorem}.

\p{Overview of the paper}
Our proof of the main result, Theorem~\ref{thm:onerelmaintheorem}, is built around two technical theorems: Theorem~\ref{thm:JSJdecomp}, which states that if $G$ is one-ended and $G\not\cong\<a, b; [a, b]^n\>$ then $\out(G)$ is virtually-cyclic, and Theorem~\ref{injecting}, which embeds $\out(G)$ into $\GL_2(\mathbb{Z})$ when $G$ is one-ended. Another key result is Lemma~\ref{lem:OutInfSplit}, which is, essentially, a rewriting result in the spirit of Moldovanskii rewriting. Lemma~\ref{lem:OutInfSplit} is notable because of its proof, which combines technical results and methods of Kapovich--Weidmann \cite{kapovich1999structure} \cite{kapovich1999two} with results of the author on the structure of the JSJ-decompositions of one-relator groups with torsion \cite{logan2014JSJ}.

In Section~\ref{sec:onerelbground} we state a key result due to Pride on Nielsen equivalence.
In Section~\ref{sec:JSJ} we apply results on fixed points of automorpisms in free groups to prove our first main technical theorem, Theorem~\ref{thm:JSJdecomp}. In Section \ref{sec:embedding} we prove our second main technical theorem, Theorem~\ref{injecting}, and we prove that if $G\cong \langle a, b; [a, b]^n\rangle$ then $\out(G)\cong \GL_2(\mathbb{Z})$.
In Section~\ref{sec:outinfinite} we determine the possibilities for $\out(G)$ if $\out(G)$ is infinite, $G$ is one-ended and $G\not\cong \langle a, b; [a, b]^n\rangle$.
In Section~\ref{sec:outfinite} we determine the possibilities for $\out(G)$ if $\out(G)$ is finite and $G$ is one-ended. 
In Section~\ref{sec:proofofmaintheorem} we assemble the proof of Theorem~\ref{thm:onerelmaintheorem} from the previous sections. In Section~\ref{sec:applications} we prove Corollary~\ref{corol:MainCorollary} and we explain how to obtain a presentation for $\aut(G)$.

\section{Nielsen equivalence of generating pairs}
\label{sec:onerelbground}

In this section we state Proposition~\ref{Pride}, which classifies the Nielsen equivalence classes of generating pairs in a two-generator, one-relator group with torsion.
We begin by motivating our use of Proposition~\ref{Pride}.

\p{The inducing homomorphism}
Proposition~\ref{Pride} is due to Pride, who used it to solve the isomorphism problem for two-generator, one-relator groups with torsion~\cite{Pride1977}. Proposition~\ref{Pride} implies that the automorphisms of a one-ended group $G=\<a, b;R^n\>$ are induced by automorphisms of the ambient free group $F(a, b)$, so there exists a homomorphism, the \emph{inducing homomorphism}, from a subgroup $H$ of $\out(F(a, b))$ to $\out(G)$, $\theta: H\twoheadrightarrow\out(G)$. The two main technical results of this paper, Theorems~\ref{thm:JSJdecomp}~and~\ref{injecting}, both apply this inducing homomorphism in fundamental ways.

\p{Nielsen equivalence}
A \emph{Nielsen transformation $\phi_t$ of the pair $(a, b)$}, or \emph{of $F(a, b)$}, is an automorphism of the free group $F(a, b)$. Two generating pairs $(y_1, y_2)$ and $(z_1, z_2)$ of a group $G=\< x_1, x_2; \mathbf{r}\>$ are \emph{Nielsen equivalent} if there exists some Nielsen transformation $\phi_t$ of the pair $(x_1, x_2)$ such that if $\phi_t(x_1)=w_1(x_1, x_2)$ and $\phi_t(x_2)=w_2(x_1, x_2)$ then the word $w_i(y_1, y_2)$ is equal in $G$ to the word $z_i$, for $i=1, 2$.
\[
(w_1(y_1, y_2), w_2(y_1, y_2)) =_G (z_1, z_2)
\]
The equivalence classes of this equivalence relation are called \emph{Nielsen equivalence classes (of generating pairs)}.

\begin{proposition}[Pride, 1977]\label{Pride}
Let $G = \< a, b; R^n\>$ with $n>1$ and $R$ not a proper power.
Suppose that $R$ is not a primitive element of $F(a, b)$, or $R$ is a primitive element and $n=2$. Then $G$ has only one Nielsen equivalence class of generating pairs. Suppose, on the other hand, that $R$ is primitive and $n>2$. Then $G$ has $\frac{1}{2}\varphi(n)$ Nielsen equivalence classes (where $\varphi$ is the Euler totient function).
\end{proposition}

\section{$\out(G)$ is virtually cyclic}
\label{sec:JSJ}

In this section we prove Theorem~\ref{thm:JSJdecomp}. Our proof applies results of Bogopolski on fixed points of automorphisms of free groups \cite{bogopolski2000classification}. We apply Theorem~\ref{thm:JSJdecomp} in Section~\ref{sec:outinfinite}, where we combine Theorems~\ref{thm:JSJdecomp}~and~\ref{injecting} to completely classify the possibilities for $\out(G)$ when $\out(G)$ is infinite and $G\not\cong\<a, b; [a, b]^n\>$.

We shall write $\gamma_g$ to mean conjugation by $g$, so $\gamma_g: h\mapsto h^g$, while $\epsilon$ denotes $1$ or $-1$, $\epsilon=\pm1$.

\begin{theorem}\label{thm:JSJdecomp}
Let $G=\<a, b; R^n\>$, with $n>1$. Suppose that $G$ is one-ended. Then either $\out(G)$ is virtually-cyclic or $G\cong \< a, b; [a, b]^n\>$.
\end{theorem}

\begin{proof}
Every automorphism $\phi$ of $G=\<a, b; R^n\>$ is induced by a Nielsen transformation $\phi_t$ of the ambient free group $F(a, b)$, by Proposition~\ref{Pride}, and such a map $\phi_t$ sends $R$ to a conjugate of $R$ or of $R^{-1}$ in $F(a, b)$ \cite[Theorem~N5]{mks}.

Suppose $R$ is not a conjugate of $[a, b]$ or $[a, b]^{-1}$ (equivalently, $G\not\cong \< a, b; [a, b]^n\>$ \cite{Pride1977}). Define the subgroup ${\stab_{0}(R)}$ of $\out(F(a, b))$ as follows.
\[
{\stab_0(W)}:=\{\phi_t\in\aut(F(a, b)) : \:\exists g\in G \text{ s.t. }\phi_t(R)=\gamma_g(R^{\epsilon})\}
\]
By the above observation, $\out(G)$ is a homomorphic image of the following group.
\[
H_0(R):={\stab_0(R)}/({\stab_0(R)}\cap\inn(F(a, b)))
\]
We prove that $H_0(R)$ is virtually cyclic. Hence $\out(G)$ is virtually cyclic. Define
\[
\stab(R):=\{\phi\in\aut(F(a, b)) : \:\exists g\in G \text{ s.t. } \phi(R)=\gamma_g(R)\}
\]
and note that the group $H(R):=\stab(R)/(\stab(R)\cap\inn(F(a, b)))$ is virtually cyclic \cite{bogopolski2000classification}. Then ${\stab_0(R)}\cap\inn(F(a, b)={\stab(R)}\cap\inn(F(a, b)$, as $g^{-1}Rg\not\equiv R^{-1}$ unless $R$ is the empty word. Therefore, $H_0(R)/H(R)\cong{\stab_0(R)}/\stab(R)$. Then $[\stab_0(R):\stab(R)]=2$ and so $[H_0(R):H(R)]=2$, and we conclude that $H_0(R)$ is virtually cyclic, as required.
\end{proof}

Theorem~\ref{thm:JSJdecomp} can also be proven using JSJ-decompositions \cite[Section 3.1]{logan2014outer}. From this viewpoint, the omission of $G\cong\<a, b; [a, b]^n\>$ is natural because this corresponds to when $G$ is Fuchsian \cite{fine211classification}, and JSJ-decompositions yield no information about $\out(G)$ when $G$ is Fuchsian \cite{levitt2005automorphisms}.

\section{$\out(G)$ embeds into $\out(F(a, b))$}
\label{sec:embedding}

In this section we prove Theorem~\ref{injecting}, which states that if $G$ is a one-ended two-generator, one-relator group with torsion then $\out(G)$ embeds into $\GL_2(\mathbb{Z})$. We do this by proving that the inducing homomorphism, as defined in Section~\ref{sec:onerelbground}, is always an isomorphism. In Sections~\ref{sec:outinfinite}~and~\ref{sec:outfinite} we apply this embedding, along with Theorem~\ref{thm:JSJdecomp}, to determine the possibilities for $\out(G)$.

We begin with Proposition~\ref{transformationhomomorphism}. This is a variant of a classical result of Nielsen \cite[Corollary~$N4$]{mks}, and allows us to view $\out(F(a, b))$ as a group of matrices. This view is useful in the proof of Theorem~\ref{injecting}. After proving Theorem~\ref{injecting}, Proposition~\ref{transformationhomomorphism} further allows us to view the groups $\out(G)$ as groups of matrices.

\begin{proposition}\label{transformationhomomorphism}(Nielsen, 1924)
Let $\phi_t: a\mapsto A$, $b\mapsto B$ be an arbitrary Nielsen transformation of $F(a, b)$. Then the map,
\begin{align*}
\xi: \operatorname{Aut}(F(a, b))&\rightarrow GL(2, \mathbb{Z})\\
\phi_t&\mapsto \left(\begin{array}{cc}\sigma_a(A) & \sigma_b(A)\\\sigma_a(B) & \sigma_b(B)\end{array}\right)
\end{align*}
is an epimorphism, and $\operatorname{Ker}(\xi)=\operatorname{Inn}(F(a, b))$.
\end{proposition}

We now state Theorem~\ref{injecting}. The remainder of this section proves this theorem. We shall use $\sigma_x(W)$ to denote the exponent sum of the letter $x$ in the word $W$, and, for $\phi\in\aut(H)$ with $H$ some group, we shall use $\widehat{\phi}$ to denote the element of $\out(H)$ with representative $\phi$. Finally, note that if $G=\< a, b; R^n\>$ is a one-relator group with torsion then, using Moldovanskii rewriting \cite{fine1994freiheitssatz}, the relator $R^n$ can be re-written as a cyclically reduced word $S^n$ such that $\sigma_a(S)=0$ and $G\cong\<a, b; S^n\>$.

\begin{theorem}\label{injecting}
Let $G = \langle a, b; R^n\rangle$ with $n>1$. Suppose that $G$ is one-ended. Then $\operatorname{Out}(G)$ embeds in $\operatorname{Out}(F(a, b))$, and the embedding is as follows. Rewrite the relator $R$ such that $\sigma_a(R)=0$ and such that $R$ is cyclically reduced, and let $\widehat{\phi}$ be an element of $\out(G)$ with a representative Nielsen transformation $\phi_t: a\mapsto A$, $b\mapsto B$. Then the following map gives the embedding.
\begin{align*}
\Theta: \out(G)&\rightarrow \GL(2, \mathbb{Z})\\
\widehat{\phi}&\mapsto \left(\begin{array}{cc}\sigma_a(A) & \sigma_b(A)\\\sigma_a(B) & \sigma_b(B)\end{array}\right)
\end{align*}
\end{theorem}

The following corollary is obtained by applying the fact that every automorphism of $F(a, b)$ maps $[a, b]$ to a conjugate of $[a, b]^{\pm1}$ \cite[Theorem~3.9]{mks}.

\begin{corollary}\label{gl2z}
Let $G\cong\langle a, b; [a, b]^n\rangle$ with $n >1$. Then $\operatorname{Out}(G) \cong \GL_2(\mathbb{Z})$.
\end{corollary}

The embedding $\Theta$ in Theorem~\ref{injecting} is the composition of the map $\theta_1:\out(G)\rightarrow\out(F(a, b)), \widehat{\phi}\mapsto\widehat{\phi}_t$ from Proposition~\ref{Pride}, which is the reverse of the inducing homomorphism $\theta$, with the isomorphism $\overline{\xi}:\out(F(a, b))\rightarrow \GL_2(\mathbb{Z})$ induced by the map $\xi$ from Proposition~\ref{transformationhomomorphism}. Therefore, to prove Theorem~\ref{injecting} it is sufficient to prove that the inducing homomorphism $\theta: H\twoheadrightarrow\out(G)$ is injective.

The proof of Theorem~\ref{injecting} is split into two cases: $R\in F(a, b)^{\prime}$ (Lemma~\ref{lem:injectingderived}) and $R\not\in F(a, b)^{\prime}$ (Lemma~\ref{lemma3}), where $F(a, b)^{\prime}$ denotes the derived subgroup of $F(a, b)$.

\begin{lemma}\label{lem:injectingderived}
Let $G = \langle a, b; R^n\rangle$ with $n>1$ and $R\in F(a, b)^{\prime}$. Then $\operatorname{Out}(G)$ embeds into $\operatorname{Out}(F(a, b))$ by the map given in Theorem~\ref{injecting}.
\end{lemma}

\begin{proof}
It is sufficient to prove that if $\phi_t$ is a Nielsen transformation with $\theta(\phi_t)\in \operatorname{Inn}(G)$, where $\theta$ is the inducing homomorphism, then $\phi\in \operatorname{Inn}(F(a, b))$. So, let $\phi_t$ be a Nielsen transformation of the pair $(a, b)$ with $\phi_t(a):=A$ and $\phi_t(b):=B$ and such that there exists $W\in F(a, b)$ with $a^W =_G A$ and $b^W =_G B$, and we prove that $\phi_t\in\inn(F(a, b))$.
As $a^W=A$ and $b^W=B$ in $G$ it must hold that $a^W=A\mod G^{\prime}$ and $b^W=B\mod G^{\prime}$. However, $G^{ab} = \langle a, b; [a, b]\rangle \cong \mathbb{Z} \times \mathbb{Z}$, as $R \in F(a, b)^{\prime}$. Therefore, it must hold that $\sigma_a(A)=1$ and $\sigma_b(A) = 0$, and that $\sigma_a(B)=0$ and $\sigma_b(B)=1$. Proposition~\ref{transformationhomomorphism} then implies $\phi_t\in\operatorname{Inn}(F(a, b))$, as required.
\end{proof}

We now prove Theorem~\ref{injecting} in the case when $G=\<a, b; R^n\>$ is one-ended and $R\not\in F(a, b)^{\prime}$. We begin by proving Lemma~\ref{lemma1}, which gives a description of the automorphisms of such a group $G$. The statement of Lemma~\ref{lemma1} assumes that the relator $R^n$ in $G=\<a, b; R^n\>$ is such that $\sigma_a(R)=0$, $\sigma_b(R)\neq 0$ and $R\neq b^{\epsilon}$. These assumptions are valid by Moldovanskii rewriting \cite{fine1994freiheitssatz}. Recall that if $\psi\in\aut(G)$ then $\widehat{\psi}$ denotes the element of $\out(G)$ with representative $\psi$.

\begin{lemma}\label{lemma1}
Let $G=\langle a, b; R^n\rangle$ with $n>1$, $\sigma_a(R)=0$, $\sigma_b(R)\neq 0$, and $R\neq b^{\epsilon}$ cyclically reduced. Suppose that $\psi$ is an arbitrary automorphism of $G$. Then $\widehat{\psi} \in \out(G)$ has a representative $\phi \in \widehat{\psi}$ of the following form.
\begin{align*}\phi: a&\mapsto a^{\epsilon_0}b^k\\ b&\mapsto b^{\epsilon_1}\end{align*}
\end{lemma}

\begin{proof}
Note that if $\phi: a\mapsto a^{\epsilon_0}b^k, b\mapsto b^{\epsilon_1}$ is a homomorphism then it is also an automorphism, as $G$ is Hopfian. By Proposition~\ref{Pride}, the automorphism $\psi$ can be realised as a Nielsen transformation $\psi_t$ of the pair $(a, b)$, where $\psi_t(a):\equiv A$ and $\psi_t(b):\equiv B$.

Let $\pi: G\rightarrow G^{ab}$ be the abelianisation map. The abelianisation has presentation $G^{ab} = \langle x, y; y^{m}, [x, y] \rangle$, where $x:=\pi(a)$ and $y:=\pi(b)$ while $m:=\sigma_b(R)$. Let $x^iy^{\alpha}:=\pi(A)$ and let $x^jy^{\beta}:=\pi(B)$. Then automorphisms of $G$ induce automorphisms of $G^{ab}$, so $\pi(B)$ has order $m\neq 0$. We therefore have the following.
\[(x^jy^{\beta})^{m}=1 \Rightarrow x^{mj}y^{m\beta} = 1 \Rightarrow x^{mj}=1\]
Then $j=0$ as $x$ has infinite order in $G^{ab}$, and so $\pi(B) = b^{\beta}$. Hence, $\sigma_a(B)=0$. Therefore, applying Proposition~\ref{transformationhomomorphism}, the Nielsen transformation $\psi_t$ corresponds to the following matrix of $\GL_2(\mathbb{Z})$.
\[
\left(\begin{array}{cc}\sigma_a(A) & \sigma_b(A)\\ 0 & \sigma_b(B)\end{array}\right)
\]
Hence, $|\sigma_a(A)|=1=|\sigma_b(B)|$. Taking $k:=\sigma_b(A)$, $\epsilon_0:=\sigma_a(A)$ and $\epsilon_1:=\sigma_b(B)$, the Nielsen transformation
$\phi_t: a\mapsto a^{\epsilon_0}b^k, b\mapsto b^{\epsilon_1}$
also corresponds to the matrix $M$. If two Nielsen transformations are equal modulo $\inn(F(a, b))$ then they are equal modulo $\inn(G)$, and so taking $\phi:=\phi_t$ we are done.
\end{proof}

We now apply Lemma~\ref{lemma1} to prove of the case of $R\not\in F(a, b)^{\prime}$ in Theorem~\ref{injecting}.

\begin{lemma}\label{lemma3}
Let $G=\langle a, b; R^n\rangle$ with $n>1$, $\sigma_a(R)=0$, $\sigma_b(R)\neq 0$ and $R\neq b^{\epsilon}$ cyclically reduced. Then $\operatorname{Out}(G)$ embeds in $\operatorname{Out}(F(a, b))$ by the map given in Theorem~\ref{injecting}.
\end{lemma}

\begin{proof}
It is sufficient to prove that if $\phi$ is an inner automorphism of $G$, $\phi \in \operatorname{Inn}(G)$, such that $\phi: a\mapsto a^{\epsilon_0}b^k$, $b\mapsto b^{\epsilon_1}$ then $a\equiv a^{\epsilon_0}b^k$ and $b\equiv b^{\epsilon_1}$, by Lemma~\ref{lemma1}.
So, let $\phi: a\mapsto a^{\epsilon_0}b^k$, $b\mapsto b^{\epsilon_1}$ with either $\epsilon_0\neq 1$, or ${\epsilon_1}\neq 1$, or $k\neq 0$, and assume that $\phi$ is inner, $\phi\in \operatorname{Inn}(G)$. Therefore, there exists some word $W(a, b)$ such that $a^W=_Ga^{\epsilon_0}b^k$ and $b^W=_Gb^{\epsilon_1}$. By the malnormality of the subgroup $\langle b\rangle$ \cite{newman1973soluble}, we can assume $W\equiv b^i$.

We shall now prove that $i\neq 0$ (so $W\neq_G1$). Suppose that $i=0$, so $a=_G a^{\epsilon_0}b^k$ and $b=_Gb^{\epsilon_1}$. If $\epsilon_0=-1$ then $a^2=b^k$, but $a$ has infinite order in the abelianisation while $b$ has finite order, a contradiction. Therefore, if $i=0$ then $\epsilon_0=1$ and so $b^{\epsilon_1}=_G b$ and $b^k=_G 1$. Now, $b$ has infinite order \cite[Corollary 4.11]{mks}, and so $\epsilon_1=1$ and $k=0$, a contradiction. Thus, $i\neq0$.

As $i\neq 0$ we have that $b^{-i}ab^i=_Ga^{\epsilon_0}b^{k}$, so $a^{\epsilon_0}b^{k-i}a^{-1}b^i=_G1$. We shall prove that $a^{\epsilon_0}b^{k-i}a^{-1}b^i$ cannot represent the trivial word in $G$, which is a contradiction and so proves the lemma. Note that if a word $U(a, b)=_G1$ then $\sigma_a(U)=0$, because the order of $a$ under the abelianisation map is infinite. Thus, $\sigma_a(a^{\epsilon_0}b^{k-i}a^{-1}b^i)=0$ and so $\epsilon_0=1$.
Moreover, $i\neq k$ as the generator $b$ has infinite order. Hence, $\epsilon_0=1$ and $i\neq k$.

So, writing $j:=k-i$, we have that $ab^ja^{-1}b^i=_G1$, for $i, j\neq 0$. However, by the Newman-Gurevich Spelling Theorem~\cite{HP}, a word of this form cannot represent the trivial word in a group $G=\<a, b;R^n\>$ under the restrictions of this lemma.
\end{proof}

We now assemble the proof of Theorem~\ref{injecting}.

\begin{proof}[Proof of Theorem~\ref{injecting}.]
Let $G=\<a, b; R^n\>$. By applying Moldovanskii rewriting, we can assume that $\sigma_a(R)=0$ \cite{fine1994freiheitssatz}. If $R\in F(a, b)$ then the result follows from Lemma~\ref{lem:injectingderived}. If $R\not\in F(a, b)^{\prime}$ then the result follows from Lemma~\ref{lemma3}.
\end{proof}

\section{The possibilities for $\out(G)$ when it is infinite}
\label{sec:outinfinite}

In this section we work under the assumption that $G$ is a one-ended two-generator, one-relator group with torsion with $\out(G)$ infinite and determine, in Theorem~\ref{thm:possifoutinfinite}, the possible isomorphism classes for $\out(G)$. We prove, in Lemma~\ref{lem:genInfiniteCase}, that every possibility occurs. In this section we further assume that $G\not\cong \langle a, b; [a, b]^n\rangle$, as this case was dealt with in Corollary~\ref{gl2z}.

Throughout this section we use Theorem~\ref{injecting} to view the relevant outer automorphism groups as matrix groups.

\subsection{The form of outer automorphisms when $\out(G)$ is infinite}
\label{sec:formofAutsforOutInf}

We begin by proving, in Lemma~\ref{lem:OutInfSplit}, that the relator $R$ can be rewritten in a particularly nice way (in Section~\ref{sec:applications} we show that this rewriting can be made algorithmic). This new form for the relator will allow us to view automorphisms as Nielsen transformations: previously, by Proposition~\ref{Pride}, we only knew that such a view existed. Note that under the assumptions of this section $\out(G)$ is virtually-$\mathbb{Z}$, by Theorem~\ref{thm:JSJdecomp}.

The proof of Lemma~\ref{lem:OutInfSplit} is based on arguments of Kapovich-Weidmann \cite{kapovich1999structure}.

\begin{lemma}
\label{lem:OutInfSplit}
Let $G=\<a, b; R^n\>$ with $n>1$. Suppose that $G\not\cong \langle a, b; [a, b]^n\rangle$ and that $G$ is one-ended. Then the following equivalence holds.
\[\out(G)\textnormal{ is virtually-}\mathbb{Z}\Longleftrightarrow G\cong\langle a, b; S^n(a^{-1}ba, b)\rangle\]
\end{lemma}

\begin{proof}
Suppose that $G\cong \langle a, b; S^n(a^{-1}ba, b)\rangle$. Then the map $a\mapsto ab$, $b\mapsto b$ is an automorphism of $G$, and it has infinite order by Theorem~\ref{injecting}. As $G\not\cong \langle a, b; [a, b]^n\rangle$, Theorem~\ref{thm:JSJdecomp} implies that $\out(G)$ is virtually-$\mathbb{Z}$.

Suppose that $\out(G)$ is virtually-$\mathbb{Z}$. Then $G$ splits as an HNN-extension or free product with amalgamation with vertex groups having finite center and edge groups virtually cyclic with infinite center \cite[Theorem~$1.4$]{levitt2005automorphisms}. The edge groups must be infinite cyclic \cite{karrass1971subgroups}, so $G=H\ast_{C_X^t=C_Y}$ or $G=H\ast_{C}K$ with $C_X, C_Y, C\cong\mathbb{Z}$ and $H, K\not\cong\mathbb{Z}$. We shall prove that $G$ must split as an HNN-extension $H\ast_{C_X^t=C_Y}$ where the base group $H$ is two-generated, which implies that the subgroup $H$ is a two-generator, one-relator group with torsion \cite{pride1977twogensubgps}. This implies that $G$ is isomorphic to a group of the required form.

Suppose that $G$ splits as a free product with amalgamation $G=H\ast_{C}K$ with $C\cong\mathbb{Z}$ and $H, K\not\cong\mathbb{Z}$, and we shall prove that either $H$ or $K$ is infinite cyclic, a contradiction. Begin by noting that $H$ and $K$ are both hyperbolic \cite[Theorem~6]{kharlampovich1998hyperbolic}, and that the amalgamating subgroup $C$ is malnormal in $H$ or $K$ \cite[Corollary~2]{kharlampovich1998hyperbolic}. We suppose, without loss of generality, that $C$ is malnormal in $K$. Now, $C$ is contained in a malnormal infinite cyclic subgroup $H_0$ of $G$ \cite[Lemma~5.5]{logan2014JSJ}.
Note that $H_0\leq H$ because $C$ is malnormal in $K$. As $H$ is hyperbolic, the amalgamating subgroup $C$ is contained in a unique maximal, virtually-cyclic subgroup of $H$, and this must be $H_0$. We shall now prove that $H_0=H$, which is our required contradiction as $H_0\cong\mathbb{Z}$ but $H\not\cong\mathbb{Z}$. Assume otherwise, so $H_0\lneq H$, and we shall look for a contradiction. Then $G$ can be written
$G=H\ast_{H_0}K_0$ where $K_0=H_0\ast_CK$ is the group generated by $H_0$ and $K$. Then $H_0$ is malnormal in $H$ while its image is malnormal in $K_0$ \cite[Lemma~3.1.10]{logan2014outer}.%
\footnote{Kapovich--Weidmann \cite{kapovich1999structure} use this result without proof. The citation is for completeness, and the proof is a straightforward application of results on conjugation \cite[Theorem~4.6]{mks} and commutativity \cite[Theorem~4.5]{mks} in free products with amalgamation.}
However, $G$ is a one-ended, two-generated group and so cannot have the form $P\ast_QR$ where $Q$ is malnormal in both $P$ and $R$ \cite[Theorem~6]{karrass1971free}, a contradiction. Hence, $H=H_0$, a contradiction. We conclude that $G$ does not split as $H\ast_CK$ where $C\cong\mathbb{Z}$ and $H, K\not\cong\mathbb{Z}$.

Therefore, $G\cong H\ast_{C_X^t=C_Y}$ with $C_X, C_Y\cong\mathbb{Z}$. It is now sufficient to prove that the base group $H$ is two-generated. Begin by noting that the associated subgroups $C_X$ and $C_Y$ are subgroups of malnormal, infinite cyclic subgroups of $G$ \cite[Lemma~5.1]{logan2014JSJ}, and indeed that one of $C_X$ or $C_Y$ is malnormal in $H$ and $C_X\cap C_Y^g$ is trivial for all $g\in H$ \cite[Corollary~1]{kharlampovich1998hyperbolic}. Therefore, $G=\<H, t; t^{-1}x^mt=y\>$ where $\<x\>$ and $\<y\>$ are malnormal in $H$. This all implies that there exists some $h\in H$ such that $G=\<th, x\>$ \cite[Corollary~3.1]{kapovich1999two}. 

Write $d=h^{-1}yh$ and $s=th$, so $\<s, x\>=G$ and we shall prove that $\<x, d\>=H$. Note that $\< x, d\>\leq H$, so assume that it is a proper subgroup and look for a contradiction. So, suppose that there exists $g\in H\setminus \< x, d\>$. Then $g$ can be written in terms of $x$ and $s$, as these elements generate $G$, and, moreover, $g$ can be written in terms of $x$, $s$ and $d$. Write $g$ as a word $W$ in $x$, $s$ and $d$ such that the number of occurrences of $s$ is minimal. That is, write $g$ in the following way, where $h_j\in\< x, d\>$ and $k$ is minimal (note that $h_i$ can be trivial).
\[
g=_GW(x, s, d)=h_0s^{\epsilon_1}h_1s^{\epsilon_2}h_2\ldots s^{\epsilon_k}h_k
\]
Note that $k>0$ as $g\not\in\< x, d\>$. If $W$ has a subword of the form $s^{-1}x^{lm}s$ then this can be replaced by $d^l$ to gain a word with fewer $s$-terms. Similarly, if $W$ has a subword of the form $sd^ls^{-1}$ then this can be replaced by $x^{lm}$ to gain a word with fewer $s$-terms. Therefore, by the minimality of $k$, $W$ is a reduced sequence for $g$ which contains $k>0$ $t$-terms. Applying Britton's Lemma \cite{L-S}, we have that $g\not\in H$, a contradiction. Therefore, $H=\< x, d\>$, as required.
\end{proof}

In Lemma~\ref{lemma12}, a key result in our proof of Theorem~\ref{thm:possifoutinfinite}, we prove under certain assumptions that if $\psi\in\aut(G)$ then $\widehat{\phi}\in\out(G)$ has a representative $\psi\in\widehat{\phi}$ of one of the following four forms.
\begin{align*}
\alpha_i: a&\mapsto a^{-1}b^i &\beta_i: a&\mapsto ab^i &\zeta_i: a&\mapsto a^{-1}b^i &\delta_i: a&\mapsto ab^i\\
b&\mapsto b &b&\mapsto b^{-1} & b&\mapsto b^{-1} &b&\mapsto b
\end{align*}
We shall use the labels $\alpha_i$, $\beta_i$, $\zeta_i$, and $\delta_i$ in the rest of Section~\ref{sec:formofAutsforOutInf} to refer to these forms, and we make no notational distinction between viewing these maps as Nielsen tranformations or as automorphisms of $G$. We later (Section~\ref{outgroups}) use $\alpha:=\alpha_0$, $\beta:=\beta_0$, $\zeta:=\zeta_0$ and $\delta:=\delta_1$.

\p{Virtually cyclic subgroups} We can assume that $\delta_1$ is an automorphism of $G$ by Lemma~\ref{lem:OutInfSplit}. We shall now work out all possible virtually cyclic subgroups of $\GL_2(\mathbb{Z})$ which contain the following matrix $\Delta_1$, which corresponds to the outer automorphism $\widehat{\delta}_1$ under the embedding given by Theorem~\ref{injecting}.
\[\Delta_1=\left( \begin{array}{cc}
1 & 1 \\
0 & 1 \end{array} \right)\]
Lemma~\ref{lem:matrixpowers} allows roots of the matrix $\Delta_i:=\Delta_1^i$, $i\neq 0$, to be computed. Note that $\Delta_i$ corresponds to $\widehat{\delta}_i$. Lemma~\ref{lem:matrixpowers} is easily proven by induction on the power $m$, and so the proof is left to the reader.

\begin{lemma}\label{lem:matrixpowers}
Let $A$ be a matrix from $\GL_2(\mathbb{Z})$: \[ A= \left( \begin{array}{cc}
a & b \\
c & d \end{array} \right)\]
Then $A^m$ has the following form where $x_m$, $y_m$ and $z_m$ are such that $x_m-y_m=z_m(a-d)$.
\[A^m=\left( \begin{array}{cc}
x_m & bz_m \\
cz_m & y_m \end{array} \right)\]
\end{lemma}

We apply lemma~\ref{lem:matrixpowers} to the following result, which tells us about certain virtually cyclic subgroups of $\GL_2(\mathbb{Z})$. Recall that if $\psi\in\aut(H)$ for some group $H$ then $\widehat{\psi}$ denotes the element of $\out(H)$ with representative $\psi$.

\begin{lemma}\label{lem:form1}
Let $\psi_t\in \aut(F(a, b))$ be such that $\langle \widehat{\psi}_t, \widehat{\delta_1}\rangle$ is a virtually cyclic subgroup of $\out(F(a, b))$. Then $\widehat{\psi}_t$ corresponds to $\widehat{\alpha}_i$, $\widehat{\beta}_i$, $\widehat{\zeta}_i$, or $\widehat{\delta}_i$ in $\out(F(a, b))$.
\end{lemma}

\begin{proof}
Our proof uses the equivalence of $\out(F(a, b))$ and $\GL_2(\mathbb{Z})$. Write $\Delta:=\Delta_1$ for the matrix corresponding to $\widehat{\delta}_1$ and $\Psi$ for the matrix corresponding to $\widehat{\psi}$. Using the notation of Lemma~\ref{lem:matrixpowers}, take $A:=\Psi$ and we shall use Lemma~\ref{lem:matrixpowers} to prove that $c=0$. This is sufficient as then $|a|=1=|d|$, by looking at the determinant of $\Psi$, which implies that $\widehat{\psi}$ is of the required form.

Suppose that $\Psi$ has infinite order. Then $\Psi^j=\Delta^k$ for some $j, k\neq 0$. Therefore, $x_j=1=y_j$, $cz_j=0$ and $bz_j=k\neq 0$. Thus, $z_j\neq 0$ and so $c=0$, as required.

Suppose that $\Psi$ has finite order. Then, $\Psi\Delta\Psi^{-1}$ has infinite order and so $\Psi\Delta^j\Psi^{-1}=\Delta^k$ for some $j, k\neq 0$. We have the following, for $\epsilon:=\operatorname{det}(\Psi)=\pm1$.
\begin{align*}
\Psi\Delta^j\Psi^{-1}
&=\epsilon\left( \begin{array}{cc}
a & b \\
c & d \end{array} \right)
\left( \begin{array}{cc}
1 & j \\
0 & 1 \end{array} \right)
\left( \begin{array}{cc}
d & -b \\
-c & a \end{array} \right)\\
&=\epsilon\left( \begin{array}{cc}
ad-bc-jac & ja^2 \\
-jc^2 & jac+ad-bc \end{array} \right)\\
&=\Delta^k
\end{align*}
Then because $j\neq 0$ we have that $c=0$, as required.
\end{proof}

Combining Lemmas~\ref{lem:OutInfSplit}~and~\ref{lem:form1} gives the following result.

\begin{lemma}\label{lemma12}
Let $G\cong\langle a, b; S^n(a^{-1}ba, b)\rangle$ with $n>1$. Suppose that $G\not\cong \< a, b; [a, b]^n\>$ and that $G$ is one-ended. If $\psi\in\aut(G)$ then $\widehat{\psi}$ corresponds to $\widehat{\alpha}_i$, $\widehat{\beta}_i$, $\widehat{\zeta}_i$, or $\widehat{\delta}_i$.
\end{lemma}

\subsection{The possibilities}
\label{outgroups}
We now prove Theorem~\ref{thm:possifoutinfinite}, which gives the possible isomorphism classes for $\out(G)$ under the restrictions of this section, and Lemma~\ref{lem:realiseinfinite}, which gives explicit examples for each of the possible generating sets in Theorem~\ref{thm:possifoutinfinite}.

First we determine, in a certain sense, the possible generators for $\out(G)$. We use the notation $\delta:=\delta_1$, $\alpha:=\alpha_0$, $\beta:=\beta_0$ and $\zeta:=\zeta_0$ for these generators.

\begin{lemma}
\label{lem:genInfiniteCase}
Let $G=\langle a, b\rangle$ be an arbitrary two-generator group. Suppose that $\delta\in\aut(G)$. Then the following hold for all $i\in\mathbb{Z}$.
\begin{enumerate}
\item If $\alpha_i\in\aut(G)$ then $\alpha\in\aut(G)$.
\item If $\beta_i\in\aut(G)$ then $\beta\in\aut(G)$.
\item If $\zeta_i\in\aut(G)$ then $\zeta\in\aut(G)$.
\end{enumerate}
\end{lemma}

\begin{proof}
The result holds as $\alpha_i=\delta^i\alpha$, $\beta_i=\delta^{-i}\beta$, and $\zeta_i=\delta^{-i}\zeta$.
\end{proof}

Theorem~\ref{thm:possifoutinfinite} gives the possible isomorphism classes when $\out(G)$ is virtually-$\mathbb{Z}$.

\begin{theorem}
\label{thm:possifoutinfinite}
Let $G=\<a, b; R^n\>$, with $n>1$. Suppose that $G\not\cong\<a, b; [a, b]^n\>$ and that $G$ is one-ended. If $\out(G)$ is infinite then $\out(G)$ is isomorphic to one of $D_{\infty}\times C_2$, $D_{\infty}$, $\mathbb{Z}\times C_2$ or $\mathbb{Z}$.

Moreover, there are five choices of generating set for $\out(G)$: we always have $\delta\in\operatorname{Aut}(G)$, and either none of, one of or all three of $\alpha$, $\beta$ and $\zeta$ are automorphisms of $G$. The following isomorphisms then hold.
\begin{enumerate}
\item\label{genset1} If $\alpha, \beta, \zeta \not\in \operatorname{Aut}(G)$ then $\operatorname{Out}(G) \cong \mathbb{Z}$.
\item\label{genset2} If $\alpha\in \operatorname{Aut}(G)$ but $\beta, \zeta \not\in \operatorname{Aut}(G)$ then $\operatorname{Out}(G) \cong D_{\infty}$.
\item\label{genset3} If $\beta\in \operatorname{Aut}(G)$ but $\alpha, \zeta \not\in \operatorname{Aut}(G)$ then $\operatorname{Out}(G) \cong D_{\infty}$.
\item\label{genset4} If $\zeta \in \operatorname{Aut}(G)$ but $\alpha, \beta \not\in \operatorname{Aut}(G)$ then $\operatorname{Out}(G) \cong \mathbb{Z} \times C_2$.
\item\label{genset5} If $\alpha, \beta, \zeta \in \operatorname{Aut}(G)$ then $\out(G)\cong D_{\infty} \times C_2$.
\end{enumerate}
\end{theorem}

\begin{proof}
By Lemma~\ref{lem:OutInfSplit}, we can assume that $R\in \langle a^{-1}ba, b\rangle$, so $\delta\in\aut(G)$. Lemmas~\ref{lemma12}~and~\ref{lem:genInfiniteCase} then give the generating sets (\ref{genset1})--(\ref{genset5}). The isomorphisms are easily computed using the embedding of Theorem~\ref{injecting}, which proves the theorem.
\end{proof}

The following lemma, Lemma~\ref{lem:realiseinfinite}, implies that each of the possible groups and generating sets in Theorem~\ref{thm:possifoutinfinite} occur. The examples in Lemma~\ref{lem:realiseinfinite} can be verified by checking whether or not $\phi(R)$ is freely conjugate to $R^{\pm 1}$ for each $\phi\in\{\alpha, \beta,\zeta\}$. This works because none of the maps $\alpha$, $\beta$ or $\zeta$ change the length of the relator $R$ so we can apply the Newman--Gurevich Spelling Theorem~\cite{HP}.
\begin{lemma}
\label{lem:realiseinfinite}
For each $Q\in\{D_{\infty}\times C_2, D_{\infty}, \mathbb{Z}\times C_2, \mathbb{Z}\}$ there exists a group $G=\langle a, b; R^n\>$, $n>1$, with $\out(G)\cong Q$. The following groups give explicit examples.
\begin{enumerate}
\item If $G = \<a, b; (aba^{-1}b^2ab^3a^{-1}b^4)^n\>$ then $\alpha, \beta, \zeta \not\in \operatorname{Aut}(G)$ and $\out(G)\cong\mathbb{Z}$.
\item If $G = \<a, b; (aba^{-1}b^2ab^3a^{-1}bab^2a^{-1}b^3)^n\>$ then $\alpha\in \operatorname{Aut}(G)$ but $\beta, \zeta \not\in \operatorname{Aut}(G)$ and $\out(G)\cong D_{\infty}$.
\item If $G = \<a, b; (aba^{-1}b^2)^n\>$ then $\beta\in \operatorname{Aut}(G)$ but $\alpha, \zeta \not\in \operatorname{Aut}(G)$ and $\out(G)\cong D_{\infty}$.
\item If $G = \<a, b; (aba^{-1}b^2ab^2a^{-1}bab^3a^{-1}b^3)^n\>$ then $\zeta \in \operatorname{Aut}(G)$ but $\alpha, \beta \not\in \operatorname{Aut}(G)$ and $\out(G)\cong\mathbb{Z}\times C_2$.
\item If $G = \<a, b; (aba^{-1}b)^n\>$ then $\alpha, \beta, \zeta \in \operatorname{Aut}(G)$ and $\out(G)\cong D_{\infty}\times C_2$.
\end{enumerate}
\end{lemma}

\section{The possibilities for $\out(G)$ when it is finite}
\label{sec:outfinite}

In this section we work under the assumption that $G$ is a one-ended two-generator, one-relator group with torsion with $\out(G)$ finite and we determine the possible isomorphism classes for $\out(G)$. We prove that every possibility occurs.

Suppose that $G$ is one-ended with $\out(G)$ finite. Then $\out(G)$ is isomorphic to a finite subgroup of $\GL_2(\mathbb{Z})$, by Theorem \ref{injecting}, and hence to a subgroup of $D_6$ or of $D_4$ \cite{zimmermann1996finite}, where $D_{n}$ denotes the dihedral group of order $2n$. We now prove that six of these nine groups can be realised in this way. The remaining three groups are dealt with in Lemma~\ref{lem:examplesFiniteOut}, which we state but do not prove.

\begin{lemma}
\label{lem:realisefinitebatch1}
For each $Q\in \{D_6, D_4, D_3, C_6, C_4, C_3\}$ there exists a group $G=\langle a, b; R^n\rangle$, $n>1$, with $\out(G)\cong Q$. The following groups give explicit examples.
\begin{enumerate}
\item If $G=\langle a, b; (a^2bab^2a^{-2}b^{-1}a^{-1}b^{-2})^n\rangle$ then $\out(G)\cong D_6$.
\item If $G= \langle a, b; [a^2, b^2]^n\rangle$ then $\out(G)\cong D_4$.
\item If $G= \langle a, b; (a^2(ab)^{-2}b^2)^n\rangle$ then $\out(G)\cong D_3$.
\item If $G= \langle a, b; R^n\rangle$ where $R$ is the word
\begin{align*}
a^2&b^3aba^{-1}b^{-2}ababa^2ba^{-1}b^{-1}a^{-1}b^{-1}a^2b^{-1}ab\\
&a^{-2}b^{-3}a^{-1}b^{-1}ab^2a^{-1}b^{-1}a^{-1}b^{-1}a^{-2}b^{-1}ababa^{-2}ba^{-1}b^{-1}
\end{align*}
then $\out(G)\cong C_6$.
\item If $G= \langle a, b; (ab^2aba^{-2}ba^{-1}b^{-2}a^{-1}b^{-1}a^2b^{-1})^n\rangle$ then $\out(G)\cong C_4$.
\item If $G= \langle a, b; (ab^{-1}a^2b^{-1}a^{-2}b^{-1}a^{-1}b^{-1}a^{-1}bab^3)^n\rangle$ then $\out(G)\cong C_3$.
\end{enumerate}
\end{lemma}

\begin{proof}
We shall use the fact that if $\out(G)$ is infinite and $G\not\cong\langle a, b; [a, b]^n\rangle$ then any finite order elements of $\out(G)$ have order two, which follows from Lemma~\ref{thm:possifoutinfinite}. We also use the fact that none of the groups $G$ in the statement of the lemma are isomorphic to $\<a, b; [a, b]^n\>$, as $\<a, b; R^n\>\cong \<a, b; [a, b]^n\>$
if and only if
the word $R$ is mapped to $[a, b]^{\epsilon}$ under some Nielsen transformation of $F(a, b)$ \cite{Pride1977},
if and only if
$R$ is freely conjugate to $[a, b]^{\pm1}$ \cite[Theorem~3.9]{mks}. Note that by Theorem~\ref{injecting}, all interactions between outer automorphisms can be verified by viewing them as elements of $\GL_2(\mathbb{Z})$.
\pf{The $D_6$ case} Note that the maps $\phi: a\mapsto b^{-1}, b\mapsto ab$ and $\psi: a\mapsto ab$, $b\mapsto b^{-1}$ define automorphisms of $G= \langle a, b; (a^2bab^2a^{-2}b^{-1}a^{-1}b^{-2})^n\rangle$. Now, $\out(G)$ is finite because $\widehat{\phi}$ has order six but $G\not\cong\<a, b;[a, b]^n\>$. Therefore, $\out(G)$ is isomorphic to either $D_6$ or $C_6$. Then, as $\widehat{\phi}^3\neq\widehat{\psi}$ but $\widehat{\psi}$ has order two, we conclude that $\out(G)\cong D_6$, as required.
\pf{The $D_4$ case} This is similar to the $D_6$ case, using the maps $\phi: a\mapsto b, b\mapsto a^{-1}$ and $\psi: a\mapsto b$, $b\mapsto a$
\pf{The $D_3$ case} Note that the maps $\phi: a\mapsto a^{-1}b^{-1}, b\mapsto a$ and $(\beta_1=)\psi: a\mapsto ab$, $b\mapsto b^{-1}$ define automorphisms of $G= \langle a, b; (a^2(ab)^{-2}b^2)^n\rangle$. As in the $D_6$ case, $\out(G)$ is finite because $\widehat{\phi}$ has order three but $G\not\cong\<a, b;[a, b]^n\>$. Noting that $\widehat{\psi}$ has order two, $\out(G)$ contains an element of order three and an element of order two, and so is isomorphic to one of $D_6$, $C_6$ or $D_3$. Suppose $\out(G)$ contains an element of order six, and consider the embedding of $\out(G)$ in $\GL_2(\mathbb{Z})$. The matrices $\Omega_1$ and $\Omega_2$ are the only matrices which satisfy the relation $\Omega^2=\Phi$, where the matrix $\Phi$ is the image of $\widehat{\phi}$ in $\GL_2(\mathbb{Z})$, and so one of these must have order six.
\[
\Omega_1=\left(\begin{array}{cc}
0 & -1\\
1&1
\end{array}
\right)
\:\:\:\:
\Omega_2=\left(\begin{array}{cc}
0 & 1\\
-1&-1
\end{array}
\right)
\]
However, $\Omega_2=\Phi^{-1}$ has order three while any Nielsen transformation which corresponds to $\Omega_1$ does not preserve the relation of the group, $(a^2(ab)^{-2}b^2)^n$, and so does not correspond to an automorphism of $G$, a contradiction. Thus, $\out(G)$ contains no element of order six and so $\out(G)\cong D_3$, as required.
\pf{The $C_{6}$ case} This is similar to the $D_3$ case, using the map $\phi: a\mapsto b^{-1}, b\mapsto ab$. We suppose that $\out(G)\cong D_6$, and obtain six matrices, $\Psi_1, \ldots, \Psi_6$, which are of order two and which satisfy the relator $(\Phi\Psi_i)^2$, where $\Phi$ is the image of $\widehat{\phi}$ in $\GL_2(\mathbb{Z})$. However, any Nielsen transformation which corresponds to one of these six matrices does not preserve the relator of the group, and so does not define an automorphism of $G$, a contradiction.
\pf{The $C_{4}$ case} This is similar to the $D_3$ case, using the map $\phi: a\mapsto b, b\mapsto a^{-1}$.
\pf{The $C_{3}$ case} This is similar to the $D_3$ case, using the map $\phi: a\mapsto a^{-1}b^{-1}, b\mapsto a$. Note that we have to eliminate each of $D_6$, $C_6$ and $D_3$.
\end{proof}

The remaining three finite subgroups of $\GL_2(\mathbb{Z})$ do each occur as $\out(G)$ with $G$ one-ended. Indeed, the following lemma, Lemma~\ref{lem:examplesFiniteOut}, gives explicit examples. We do not prove Lemma~\ref{lem:examplesFiniteOut} here as the proof is disproportionately long. The issue is that the three remaining groups all embed into $D_{\infty}\times C_2$ (hence the methods used to prove Lemma~\ref{lem:realisefinitebatch1} will not work).
The result may be verified by applying an algorithm of Dahmani-Guirardel \cite[Theorem~3]{dahmani2011isomorphism}, while the author has proven Lemma~\ref{lem:examplesFiniteOut} using a different method in his PhD thesis \cite[Section 3.4.1]{logan2014outer}.

\begin{lemma}
\label{lem:examplesFiniteOut}
For $Q$ each of the groups $C_2\times C_2$, $C_2$ and the trivial group there exists a group $G=\langle a, b; R^n\rangle$, $n>1$, with $\out(G)\cong Q$. The following groups give explicit examples.
\begin{enumerate}
\item If $G= \langle a, b; (a^2ba^{-2}b)^n\rangle$ then $\out(G)\cong C_2\times C_2$.
\item If $G= \langle a, b; (a^2ba^{-3}b)^n\rangle$ then $\out(G)\cong C_2$.
\item If $G= \langle a, b; (a^{-2}ba^{4}ba^{-3}ba^{5}b)^n\rangle$ then $\out(G)$ is trivial.
\end{enumerate}
\end{lemma}

\section{Assembling the proof of Theorem~\ref{thm:onerelmaintheorem}}
\label{sec:proofofmaintheorem}

We now assemble the proof of Theorem~\ref{thm:onerelmaintheorem}, as stated in the introduction. We have so far omitted the case of $G=\< a, b; R^n\>$ with $R$ primitive, which occurs precisely when $G$ has more than one end, and precisely when $G\cong\mathbb{Z}\ast C_n$.
Proposition~\ref{primitivethm} now deals with this case.
\begin{proposition}\label{primitivethm}
Let $G= \< a, b; R^n\>$ with $n>1$. Suppose that $R$ is a primitive element of $F(a, b)$. Then the following isomorphism holds, where $\aut(C_n)$ commutes with the flip generator of $D_n$ and acts on the rotation generator in the natural way as automorphisms of $C_n$.
\[\out(G) \cong D_{n} \rtimes \aut(C_n)\] 
\end{proposition}
We do not prove Proposition~\ref{primitivethm}, as here $G\cong \mathbb{Z}\ast C_n$. There are a number of ways to approach the outer automorphism group of such a free product, and indeed presentations for the automorphism groups of free products are known \cite{fouxe1940uber} \cite{gilbert1987presentations} (of course the presentation of an automorphism group does not necessarilly yield the presentation of the corresponding outer automorphism group). A sketch proof of Proposition~\ref{primitivethm} can be found in the author's PhD thesis \cite[Section 3.5]{logan2014outer}.

\begin{proof}[Proof of Theorem~\ref{thm:onerelmaintheorem}.]
Let $G=\<a, b; R^n\>$ with $n>1$ and $R$ not a proper power.

Suppose that $G\cong\<a, b; [a, b]^n\>$. Then $\out(G)\cong \GL_2(\mathbb{Z})$ by Corollary~\ref{gl2z}.

Suppose that $G$ is one-ended, not isomorphic to $\<a, b; [a, b]^n\>$, and $\out(G)$ is infinite. Then $\out(G)$ is one of $D_{\infty}\times C_2$, $D_{\infty}$, $\mathbb{Z}\times C_2$ or $\mathbb{Z}$, by Lemma~\ref{thm:possifoutinfinite}. By Lemma~\ref{lem:realiseinfinite}, all these possibilities occur.

Suppose that $G$ is one-ended, not isomorphic to $\<a, b; [a, b]^n\>$ and $\out(G)$ is finite. Then $\out(G)$ is isomorphic to a finite subgroup of $\GL_2(\mathbb{Z})$, and the finite subgroups of $\GL_2(\mathbb{Z})$, up to isomorphism, are precisely the subgroups of $D_4$ and of $D_6$ \cite{zimmermann1996finite}, as required. By Lemma~\ref{lem:realisefinitebatch1} and Lemma~\ref{lem:examplesFiniteOut}, all these possibilities occur.

Suppose that $G$ is not one-ended. Then $R$ is a primitive element of $F(a, b)$ \cite{Pride1977} \cite{fischer1972one}, and then $\out(G)\cong D_n\rtimes\aut(C_n)$ by Proposition~\ref{primitivethm}.
\end{proof}

\section{Two applications}
\label{sec:applications}

We conclude this paper by proving Corollary~\ref{corol:MainCorollary} from the introduction and by describing how to give a presentation of
$\aut(G)$ for $G$ one-ended.

\p{Proof of Corollary~\ref{corol:MainCorollary}}
Recall that Corollary~\ref{corol:MainCorollary} states the existence of an algorithm which takes as input a presentation $\langle a, b; R^n\rangle$, $n>1$, defining a group $G$ and gives as output the isomorphism class of $\out(G)$.
\begin{proof}[Proof of Corollary~\ref{corol:MainCorollary}]
To prove Corollary~\ref{corol:MainCorollary} we give the relevant algorithm. Our algorithm applies the facts that there exists an algorithm to rewrite a word $U$ in $F(a, b)$ as $V^n$ where $V$ is not a proper power in $F(a, b)$, that it is decidable if a hyperbolic group splits over a virtually-cyclic group with infinite center \cite[Proposition~6.1]{dahmani2011isomorphism}, and that there exists an algorithm to determine the generators of the outer automorphism group of a hyperbolic group \cite[Theorem~3]{dahmani2011isomorphism}. Note that if it is known that $G$ splits over a virtually-$\mathbb{Z}$ group with infinite center then $G\cong \langle a, b; S^n(a^{-1}ba, b)\rangle$ for some word $S$, by Lemma~\ref{lem:OutInfSplit}, and the word $S$ can be found by enumerating the Nielsen transformations of $F(a, b)$ and applying them to $R$ \cite{Pride1977}. Given $G=\<a, b; R^n\>$, $n>1$, our algorithm is as follows.
\begin{enumerate}
\item Rewrite $R^n$ such that $R$ is freely and cyclically reduced and not a proper power in $F(a, b)$, and such that $\sigma_a(R)=0$.
\begin{enumerate}
\item If $R=b^{\epsilon}$ then $\out(G)\cong D_{2n}\rtimes \aut(C_n)$.
\item If $R$ is empty or a cyclic shift of $[a, b]^{\pm 1}$ then $\out(G)\cong \GL_2(\mathbb{Z})$.
\end{enumerate}
\item Does the JSJ-decomposition of $G$ split?
\begin{enumerate}
\item\label{list:outinfinite} If \emph{yes} then rewrite $R$ as $S(a^{-1}ba, b)$ and apply Section \ref{sec:outinfinite} to find $\out(G)$.
\item If \emph{no} then obtain the generators for $\out(G)$ and apply Section \ref{sec:outfinite} to find $\out(G)$.%
\end{enumerate}
\end{enumerate}
\end{proof}

\p{Presentations for \boldmath{$\aut(G)$}}
We now describe, with examples, how to write down a presentation for $\operatorname{Aut}(G)$ using $\operatorname{Out}(G)$, where $G=\<a, b; R^n\>$, $n>1$, is one-ended (the infinitely-ended case has appeared in print before \cite{fouxe1940uber} \cite{gilbert1987presentations}).

To begin, note that $\inn(G)\cong G$ in the canonical way (as $G$ has trivial center \cite{baumslagtaylor}), so we have the following relation.
\begin{equation}
\gamma_{_{R^n}}=1\label{type1}
\end{equation}
Next, take a transversal $T$ for $\out(G)$ which consists of Nielsen transformations and denote by $\mathcal{O}$ a subset of this transversal which will generate $\out(G)$. This transversal $T$ exists by Theorem~\ref{injecting}. Now, we have that $\gamma_{w}^{\psi}=\gamma_{\psi(w)}$ for all $\psi\in\aut(G)$. Therefore, we have the following relations for all $\psi\in\mathcal{O}$.
\begin{equation}
\gamma_{a}^{\psi}=\gamma_{\psi(a)}, \:\:
\gamma_b^{\psi}=\gamma_{\psi(b)}
\label{type2}
\end{equation}
We now have to ascertain how the elements of $\mathcal{O}$ multiply together. So, let $\langle X; \mathbf{r}\rangle$ be a presentation for $\out(G)$ which corresponds to the generators $\mathcal{O}$, then if $S\in\mathbf{r}$ is a relator we have that
\begin{equation}
S=\gamma_w
\label{type3}
\end{equation}
is a relation in $\aut(G)$ where $\gamma_w$ is the appropriate inner automorphism.
It is clear that these three kinds of relations (\ref{type1})--(\ref{type3}) are all the relations.

We now give some examples. We write $w$ for $\gamma_w$ (so $w$ represents the automorphism corresponding to conjugation by $w$).
\begin{example}
If $\operatorname{Out}(G) = \langle \widehat{\alpha}_i\rangle$ then $\aut(G)$ is the following group. \[\operatorname{Aut}(G) = \langle \alpha_i, a, b; R^n(a, b), \alpha_i^2=b^i, a^{\alpha_i} = a^{-1}b^i, b^{\alpha_i}=b\rangle\]
\end{example}
\begin{example}
If $\operatorname{Out}(G) = \langle \widehat{\beta}_i\rangle$ then $\aut(G)$ is the following group. \begin{align*}\operatorname{Aut}(G) &= \langle\beta_i, a, b; R^n(a, b), \beta_i^2=1, a^{\beta_i}=ab^i, b^{\beta_i} = b^{-1} \rangle\\&\cong G\rtimes C_2\end{align*}
\end{example}
\begin{example}
If $G\cong \langle a, b; [a, b]^n\rangle$ and writing $\aut(F(a, b))=\langle a, b, X; \mathbf{r}\rangle$ then we have the following group. \[\aut(G)=\langle a, b, X; \mathbf{r}, [a, b]^n\rangle\]
\end{example}

\bibliographystyle{amsalpha}
\bibliography{BibTexBibliography}
\end{document}